\documentclass[a4paper,10pt]{amsart}
\usepackage[utf8]{inputenc}
\usepackage{amsthm} %important for theorem style and numbering, see explanation below
\usepackage{amsmath}
\usepackage{bm}
\usepackage{amsfonts}
\usepackage{amssymb}
\usepackage{mathtools}
\usepackage{appendix}
\usepackage{mathrsfs}
\usepackage{setspace}
\usepackage{xcolor}
\usepackage{todonotes}
\usepackage[foot]{amsaddr}
\usepackage[pdfdisplaydoctitle,colorlinks,breaklinks,urlcolor=blue,linkcolor=blue,citecolor=blue]{hyperref} 
\usepackage[capitalise,nameinlink]{cleveref}
\newcommand{\A}{\mathcal{A}}

\newcommand{\B}{\mathcal{B}}
\newcommand{\C}{\mathbb{C}}
\newcommand{\cyl}{\mathcal{C}}
\newcommand{\D}{\mathcal{D}}

\newcommand{\EE}{\mathbb{E}}
\newcommand{\F}{\mathcal{F}}

\newcommand{\M}{\mathcal{M}}
\newcommand{\N}{\mathbb{N}}
\newcommand{\Nsc}{\mathcal{N}}

\newcommand{\PP}{\mathbb{P}}

\newcommand{\R}{\mathbb{R}}
\renewcommand{\S}{\mathscr{S}}
\newcommand{\T}{\mathcal{T}}
\newcommand{\U}{\mathcal{U}}

\newcommand{\XX}{\mathcal{X}}

\let\div\relax
\DeclareMathOperator{\div}{div}

\renewcommand{\epsilon}{\varepsilon}

\let\div\relax
\DeclareMathOperator{\div}{div}

\newcommand{\de}{\partial}
\newcommand{\set}[1]{\left\{#1\right\}}
\newcommand{\pa}[1]{\left(#1\right)}
\newcommand{\bra}[1]{\left[#1\right]}

\newcommand{\norm}[1]{\left\|#1\right\|}
\newcommand{\brak}[1]{\left\langle#1\right\rangle}

\newcommand{\expt}[2][]{\mathbb{E}_{#1}\left[#2\right]}

\newtheorem{theorem}{Theorem}%[section]
\newtheorem{definition}[theorem]{Definition}
\newtheorem{hypothesis}{Hypothesis}
\newtheorem{corollary}[theorem]{Corollary}
\newtheorem{lemma}[theorem]{Lemma}
\newtheorem{proposition}[theorem]{Proposition}

\theoremstyle{remark}
\newtheorem{remark}[theorem]{Remark}
\newtheorem{open}[theorem]{Open Question}
\numberwithin{equation}{section}

\title[Independently Scattered Invariant Measures]{Infinitesimal invariance of completely Random Measures for 2D Euler Equations}
\author[F. Grotto]{Francesco Grotto}
\address{Universit\'e du Luxembourg, Maison du Nombre, 6 Avenue de la Fonte, 4364 Esch-sur-Alzette, Luxembourg}
\email{francesco.grotto@uni.lu}
\author[G. Peccati]{Giovanni Peccati}
\email{giovanni.peccati@uni.lu}
\date{31st August 2021}

\begin{document}

\begin{abstract}	
We consider suitable weak solutions of 2-dimensional Euler equations on bounded domains,
and show that the class of completely random measures is infinitesimally invariant for the dynamics.
Space regularity of samples of these random fields falls outside of the well-posedness regime
of the PDE under consideration, so it is necessary to resort to stochastic integrals
with respect to the candidate invariant measure in order to give a definition of the dynamics.
Our findings generalize and unify previous results on Gaussian stationary solutions of Euler equations
and point vortices dynamics. 
We also discuss difficulties arising when attempting to produce a solution flow for Euler's equations
preserving independently scattered random measures.
\end{abstract}

\maketitle

%%%%%%%%%%%%%%%%%%%%%%%%%%%%%%%%%%%%%%%%%%%%%%%%%%%%%%%%%%%%%%%%%%%%%%%%%%%%%%%%%%%%%%%%%%%%%
\section{Introduction}\label{sec:introduction}

``Completely random'' or ``independently scattered'' measures
(the precise definition of which we defer to subsequent sections)
are random fields indexed by collections of measurable sets, 
whose restrictions to disjoint regions of the underlying space are stochastically independent. 
These objects are by now fundamental components of modern stochastic analysis: 
for instance, completely random measures appear in the stochastic integral representation 
of stationary random fields on manifolds \cite{AdlerTaylor07, Yadrenko83, MarinucciPeccati11} 
and L\'evy processes \cite{sato}, 
and they are the axis around which stochastic differential calculus is built, 
both in a continuum \cite{Nualart1995, NouPec12} and discrete setting \cite{LastPenrose18, PeccatiTaqqu10}. 

The aim of this paper is to illustrate a distinguished use of completely random measures in the context of 
\emph{(deterministic) partial differential equations with random initial data}:
we will show how such random fields constitute natural candidates as invariant measures
for 2-dimensional Euler's equations, and {illustrate} how an appropriate notion of \emph{weak solution}
relying on stochastic integrals can be used to study the PDE dynamics preserving these random fields.
Our approach involves the simultaneous use of Gaussian and discrete 
(more precisely, compound Poisson) random measures, and is meant to extend and unify earlier works on the matter, where Gaussian and discrete initial conditions are dealt with separately.

We observe that our study will bring to the forefront an important collection 
of random variables that are naturally attached to random measures, that is, 
\emph{multiple stochastic integrals} \cite{KW, PeccatiTaqqu10}, 
that we will study in a dynamical setting. 
These objects have recently played a crucial role in the asymptotic analysis of local 
geometric quantities of random fields --- see \cite{MarinucciPeccati11, NPR19} 
for some representative applications in this direction.

\subsection{2D Euler's Equations and Heuristics on Invariant Measures}
We refer the reader to the classic monograph \cite{Marchioro1994} for a detailed 
introduction to the topics discussed in the present section.
Let $\D\subset \R^2$ be a simply connected, bounded open domain with smooth boundary.
We consider 2-dimensional Euler's equations on $\D$ in terms of the vorticity $\omega$,
\begin{equation}\label{eq:eulero}
	\begin{cases}
	\de_t \omega(t,x)+ u(t,x)\cdot \nabla \omega(t,x)=0, &x\in\D\\
	\omega(t,x)=\nabla^\perp\cdot u(t,x),\quad \div u(t,x)=0, &x\in\D\\
	u(t,x)\cdot \hat n=0, &x\in\partial\D,
	\end{cases}
\end{equation}
where the velocity field $u$ is expressed in terms of $\omega$ by means of the Biot-Savart law
\cite[Section 1.2]{Marchioro1994}, 
\begin{equation*}
	u=K\ast \omega,\quad K(x,y)=\nabla^\perp_x G(x,y)
\end{equation*}
with $\nabla^\perp=(-\de_2,\de_1)$ and $G=(-\Delta)^{-1}$ indicates the Green function 
of the Laplace's operator with Dirichlet boundary conditions.
Since $G$ has a logarithmic singularity, $K$ is a singular convolution kernel.

Equations \cref{eq:eulero} are in the form of a (nonlinear) transport equation,
driven by the divergence-less vector field $u$, and the system is well-posed
for initial data $\omega_0\in L^\infty(\D)$, see \cite{Judovic1963}.
In this regime, if $\Phi_t:\D\to\D$ is the flow of $u$
at time $t$, the evolution of vorticity is thus given by $\omega_t\circ \Phi_t=\omega_0$.
This fact provides a remarkable intuition 
for candidate invariant measures of \cref{eq:eulero},
which we now present from a mostly heuristic perspective.

Let us fix a probability space $(\Omega,\F,\PP)$,
and consider a random measure $M$, that is, $M$ a $\sigma$-additive
map from the class $\B$ of Borel sets of $\D$ to the Banach space $L^1(\Omega,\F,\PP)$,
satisfying the following:

\begin{hypothesis}\label{hyp:main} For any $A,B\in\B$,
	\begin{itemize}
		\item[a)] if $A$ and $B$ are disjoint (up to a Lebesgue-negligible set),
		the random variables $M(A)$ and $M(B)$ are independent;
		\item[b)] if $A$ and $B$ have the same Lebesgue measure, $|A|=|B|$,
		random variables $M(A)$ and $M(B)$ have the same distribution. 
	\end{itemize}
\end{hypothesis}

Under this assumption, since the flow $\Phi_t$ of the divergence-less vector field $u$ preserves the Lebesgue measure on $\D$,
we infer that the image random measure $M_t(A)=M(\Phi_{-t}A)$ has the same law of $M$.
Indeed, by \cref{hyp:main} b), for fixed $A$ we have that $M(A)$ and $M_t(A)$
have the same law, {and it is an easy exercise to show (by using the properties of $\Phi_t$) 
that the vectors $(M(A_1),\dots, M(A_n))$ and $(M_t(A_1),\dots, M_t(A_n))$
have the same distribution for every choice of measurable sets $A_1,\dots, A_n$.}

A natural heuristic claim is thus that a random vorticity distribution $\omega$, such that
its (suitably defined) integrals $\int_A \omega(x)dx=M(A)$ are a random measure as above,
is preserved by the dynamics, or in other words, its law is an invariant measure for \cref{eq:eulero}.  
However, the above reasoning only makes sense if $\Phi_t$ is regular enough,
and when vorticity is taken to be a random measure satisfying \cref{hyp:main}
its samples are not even functions, the velocity field is all but regular: in fact it is not trivial
even to give a sensible meaning to the dynamics.

A weak formulation of \cref{eq:eulero} allowing solutions with poor space regularity
has been known since the works of Delort and Schochet \cite{Delort1991,Schochet1995},
but it does not allow to define solutions starting from a given sample of a random field $M$ as above.
The next fundamental step to be considered consists in interpreting
integrals involving $\omega$ in the weak formulation as \emph{(multiple) stochastic integrals}
with respect to the random integrator $ \omega(x)dx \sim M$.
This point of view was first considered by Flandoli \cite{Flandoli2018} in the case 
where $M$ is a Gaussian measure satisfying \cref{hyp:main},
that is, when $M$ is a multiple of a \emph{white noise};
in that work it was also discussed how such a notion of solution is a sensible one,
producing approximations of Gaussian solutions with smooth solutions and scaling limits of point vortices systems.
Related works \cite{Grotto2020b, Grotto2020c} have analyzed more general Gaussian invariant measures,
also in cases where the PDE dynamics include stochastic forcing.

\subsection{Towards Equilibrium Flows}\label{ssec:flow}

Our ideal goal is now to produce a solution flow of Euler's equations \cref{eq:eulero} 
preserving a given completely random measure.
In other words, given a completely random measure $M$, we would like to accomplish the following tasks:

\begin{itemize}
\item[\bf (a)] identify $M$ with a random element taking values in a suitable space $\XX$ of generalised functions, 
in such a way that the triple $(\XX, \mathbb{X}, m)$ --- where $m$ and $\XX$ are, respectively, the law of and the $\sigma$-field generated by $M$ --- is a probability space;

\item[\bf (b)] provide a sensible notion of weak solutions of the PDE taking values in $\XX$,
exploiting the probability measure $m$ defined on it;

\item[\bf (c)] build a measure-preserving one-parameter group of transformations
$\T_t:\XX\to\XX$, $t\in\R$, whose trajectories can be identified with weak solutions of point {\bf (b)}. 
\end{itemize}

\begin{remark} We will simply write $(\XX, m)$ to denote the triple $(\XX, \mathbb{X}, m)$ evoked above, 
	tacitly assuming that the underlying $\sigma$-field is the one generated by $M$. 
\end{remark}

Fully realising task {\bf (c)} is an open problem even in the aforementioned purely Gaussian case,
in which the existence of solutions has been shown \cite{Albeverio1990,Flandoli2018} only in the following sense:
for almost any sample of Gaussian white noise $W$, there exists a continuous map $\omega\in C([0,T],\S'(\D))$,
taking values in the space $\S'(\D)$ of generalized functions on $\D$, such that
\begin{itemize}
	\item $\omega_0=W$ and, under the law of $W$, the random element $\omega_t$ has the law of $W$, for all $t\in (0,T]$ ;
	\item Euler's equation is satisfied in a suitable weak form (detailed in \cref{sec:weaksol}).
\end{itemize}
This is clearly weaker compared to existence of a solution flow.

At present, building a solution flow (in the sense of {\bf (a)}---{\bf (c)}) 
in the even broader setting of completely random measures, is out of reach. 
In this contribution, we will describe a precise strategy for tackling this objective 
by the Koopman-von Neumann approach introduced in this fluid dynamics context by
\cite{Albeverio2002,Albeverio2003a,Albeverio2008,Albeverio2010}.
Although we will fall short of fully realising our program (see Section \cref{sec:conclusion} for a detailed discussion), 
our main contribution \cref{thm:main} represents a first, essential step.

Our way of attacking the matter is to follow the steps detailed in the next five points, 
of which only the first four will be completely dealt with in the present work.

\begin{enumerate}
	\item We will first select a collection of random measures $M$ satisfying \cref{hyp:main},
	(the restriction being finiteness of moments)
	and prove in Lemma \ref{cor:randomdistr} that a suitable restriction of $M$ lives a.s. 
	in a subspace $\XX$ of generalized functions on $\D$. 
	We will then focus on the particular set of \emph{cylinder} observables associated with
	 the probability space $(\XX,m)$.
	\item In \cref{ssec:weaksol}, we will then consider a suitable weak formulation of 
	Euler's equations \cref{eq:eulero}, allowing solutions taking values
	in the phase space $\XX$: in fact, and more precisely, the notion of solution we put forward
	makes essential use of the probability measure $m$, since the nonlinearity of the PDE is too singular
	to be defined for generic elements of $\XX$, and solutions are defined only almost surely with respect
	to $m$, in terms of stochastic integrals.
	\item By Koopman's lemma \cite{Goodrich1980}, we argue in \cref{eq:eulero} that a solution flow 
	preserving $m$ would induce a strongly continuous group $\U=e^{it\A}$ of unitary transformations on $L^2(\XX,m)$, 
	and in sight of the previous point we deduce the form that the generator $\A$ of such a group 
	must take on cylinder observables.
	\item Our main result \cref{thm:main} is that the expression for $\A$ deduced in the previous point 
	defines a symmetric operator on $L^2(\XX,m)$, and that it has at least a self-adjoint extension 
	generating a strongly continuous group of unitary operators.
	\item As argued in \cref{sec:conclusion}, if one could prove that the latter group of operators 
	preserve $L^\infty(\XX,m)$, a converse of Koopman's lemma would imply that the group 
	consists of Koopman operators of a measure-preserving flow on $(\XX,m)$,
	and because of the way we define $\A$ this would be a solution flow for Euler's equation 
	in the sense of {\bf (a)}---{\bf (c)}. 
\end{enumerate}

As anticipated, Point (5) will not be completed in the present paper: 
should it be done, this would actually be the first uniqueness result for 
Euler's dynamics in a very low regularity regime.

The rest of the paper is organized according to the list of items (1)-(5).
In particular our main result, \cref{thm:main}, is proved in \cref{sec:koopman},
and in \cref{sec:conclusion} we discuss more precisely point (5),
together with a number of further remarks and generalizations. 

To conclude the section, let us recall some relevant previous studies. 
Gaussian invariant measures of \cref{eq:eulero}
were originally introduced and studied in a series of works by Albeverio and others \cite{Albeverio1979,Albeverio1990}:
in fact, the idea of Gaussian and Poissonian invariant measures as special cases of the general,
independently scattered one, dates back to \cite{Albeverio2003}.
The works \cite{Benfatto1987,Grotto2020b,Grotto2020c,Grotto2020d} concern relations, in the form of scaling limits,
between Gaussian and Poissonian random fields preserved by Euler's dynamics.

\subsection{Notation}\label{ssec:notation}
From now on, all random elements are assumed to be defined on a suitable probability space 
$(\Omega,\F,\PP)$, with $\EE$ denoting expectation with respect to $\PP$. 
The letter $\B$ stands for the Borel $\sigma$-algebra of $\D$, whereas $dx$ is the Lebesgue measure. 
When convenient, we will also use the symbol $|A|$ to indicate the Lebesgue measure of a given $A\in\B$. 
We denote by $\S=\S(\D)=C_c^\infty(\D)$ the space of smooth functions with compact support in $\D$,
and by $\S'$ the dual space of distributions\footnote{
	To avoid confusion, \emph{distribution} will always refer to a linear continuous functional of smooth functions, and never to the law of a random variable.} 
on $\mathcal D$.

The symbol $\omega$ will always be used for vorticity, never for elements of the probability space $\Omega$
( playing a role exclusively in \cref{sec:randomfields}): in fact, 
we will not need to explicitly specify any dependence on the elements of $\Omega$,
in such a way that no ambiguity will arise.

%%%%%%%%%%%%%%%%%%%%%%%%%%%%%%%%%%%%%%%%%%%%%%%%%%%%%%%%%%%%%%%%%%%%%%%%%%%%%%%%%%%%%%%%%%%%%
\section{Independently Scattered Random Measures and Stochastic Integrals}\label{sec:randomfields}

We begin with a review of concepts from probability theory.
All results and facts of this section are standard: we refer to \cite{KW,Rajput1989,Rota1997,Szulga1991}
and, in particular, to \cite{PeccatiTaqqu10} (whose setting is closest to ours).   

\begin{definition}\label{def:rm}
	A \emph{random measure} on a measurable space $(\D,\B)$ is a real-valued random field $\set{M(A):A\in\B}$,
	indexed by $\B$ and satisfying the following property: for all sequences $A_1,A_2,\dots\in\B$ of disjoint Borel sets,
	\begin{equation*}
	M\pa{\bigcup_{n=1}^\infty A_n}=\sum_{n=1}^\infty M(A_n),
	\end{equation*}
	where the series on the right-hand side converges in probability.
\end{definition}

We now give an explicit definition of the random measures we consider in the present paper: 
as explained in \cref{ssec:characterize}, up to some additional integrability assumption 
they are actually the ones characterized by \cref{hyp:main}.

\smallskip

We assume that on the probability space $(\Omega,\F,\PP)$ the following two random measures are defined 
(see \cite[Chapter 5]{PeccatiTaqqu10}, as well as \cite{LastPenrose18, NouPec12} for further details).
\begin{itemize}
	\item A \emph{Gaussian white noise} on $\D$, that is, a centered Gaussian field $W$ defined on $\B$
	with covariance
	\begin{equation*}
	\expt{W(A)W(B)}=|A\cap B|, \quad A,B\in\B.
	\end{equation*}
	\item A \emph{ compensated} compound Poisson point process $P$ independent of $W$ given by
	\begin{equation*}
		P(A)=\int_A \int_{ \R} \gamma\,  \Nsc(dx,d\gamma), \quad A\in\B,
	\end{equation*}
	where $\Nsc$ is a compensated Poisson point process independent of $W$ on the product space $\D\times \R$, 
	whose control measure is the product of the Lebesgue measure on $\D$
	and of an atomless measure $\nu$ on $\R$ such that 
	$\int_\R |\gamma| d\nu(\gamma), \int_\R |\gamma|^2 d\nu(\gamma)<\infty$.		
\end{itemize}

\begin{remark}\label{rmk:poisson}
	The integrability properties imposed on $\nu$ imply that it is a $\sigma$-finite measure on $\R$,
	possibly giving infinite mass to neighborhoods of $0$. It is a standard fact (see \cite[Corollary 6.5]{LastPenrose18}) that samples of $P$
	can be identified (without changing probability space) with infinite sums of Dirac deltas
	$\sum_{i=1}^\infty \gamma_i \delta_{x_i}$, $x_i\in\D$, $\gamma_i\in\R$ with 
	$\sum_{i=1}^\infty |\gamma_i|=\int_\R |\gamma| d\nu(\gamma)<\infty$. Notice that $\nu$ being atomless implies in particular that $\nu(\set{0})=0$, 
	which makes this representation of samples unique, since $\gamma_i$ (almost) never take value $0$.
	Finally, assuming that $P$ is compensated implies $\int_\R \gamma d\nu(\gamma)=0$.
\end{remark}

For $a\in\R$, $q\in [0,\infty)$, we introduce the random measure
\begin{equation}\label{eq:levyito}
	M(A): =a|A|+\sqrt{q}W(A)+P(A), \quad A\in \mathcal{B},
\end{equation}
and write $N \sim [a,q,\nu]$ to indicate that another random measure $N$ has the same law of $M$, 
as defined in \cref{eq:levyito}.

\begin{remark}
	In applications, it is often natural to consider measures of the type of \cref{eq:levyito} 
	with $P$ a \emph{non-compensated} compound Poisson process. 
	However, it is immediately seen that the compensation can always be forced 
	into the definition of $M$ by suitably modifying the constant $a$.
\end{remark}

In order to specify a formulation of Euler's equation \cref{eq:eulero} that is meaningful
when the (suitably defined) random mapping $A\mapsto \int_A \omega_t dx$ has the law of $A\mapsto M(A)$,
it is convenient to regard the random measures introduced above as
random variables taking values in distribution spaces.
If $M\sim [a,q,\nu]$ and $f\in L^2(\D)$, we can set
\begin{equation} \label{e:f}
	I^1_M(f)=a\int_{\D} f dx+ \sqrt{q} W(f)+ P(f),
\end{equation}
where $I^1_W(f)$ is the It\^o integral of $f$ with respect to the white noise $W$ \cite{Nualart1995, NouPec12}, and
\begin{equation*}
	I^1_P(f)=\int_\D\int_\R \gamma f(x) \Nsc(dx,d\gamma)
\end{equation*}
is the Poissonian integral of $f$ \cite{LastPenrose18}. When $f$ is an indicator function of a Borel set $A$,
\cref{e:f} corresponds to \cref{eq:levyito}.

Although $I^1_M(f)$ is well-defined as a random variable for $f\in L^2(\D)$,
it turns out that restricting the set of integrands yields fixed samples
of $M$ as distributions, \emph{i.e.} generalized functions.

\begin{lemma}\label{cor:randomdistr}
	Let $M\sim [a,q,\nu]$ be defined on the probability space $(\Omega,\F,\PP)$.
	Then, $\PP$-almost surely, the map $\S(\D)\ni f\mapsto I^1_M(f)$ is a continuous linear functional;
	in other words, we can regard $M$ as a $\S'$-valued random variable.
	
	The support of the law $m$ of $M$ in $\S'(\D)$ is contained in a Polish subspace $\XX\subset \S'(\D)$,
	so that $(\XX,m)$ is a standard Borel probability space.
\end{lemma}

In the following, $H^k_0(\D)$ is the usual Sobolev space of functions vanishing at the boundary of $\D$,
that is the closure of $\S(\D)$ under the norm $\norm{\cdot}_{H^k_0(\D)}=\norm{(1+\Delta)^{k/2}\cdot }_{L^2(\D)}$
with $\Delta$ the Dirichlet Laplacian as above.

\begin{proof}
	The constant part $a dx$ and the Poissonian one are both elements of signed Radon measures on $\D$, $\M(\D)$
	(\emph{cf.} \cref{rmk:poisson}).
	We can thus focus on the Gaussian part.
	The following argument is standard, so we only sketch it: we refer to \cite[Chapter 3]{Samorodnitsky1994}
	for a detailed discussion of representations and orthogonal expansions of Gaussian random fields.
	
	It is a standard fact that the It\^o integral $L^2(\D)\ni f\mapsto I^1_W(f)\in L^2(\PP)$ 
	defines an isometry of Hilbert spaces.
	It is then a general fact (see \cite{DaPrato1992}) that a Gaussian measure (in our case $W$) 
	whose reproducing kernel Hilbert space $H$ is Hilbert-Schmidt embedded in another Hilbert space $E$, 
	can be identified as a Gaussian random element of $E$. 
	We then conclude observing that the embedding $L^2(\D)\hookrightarrow H^2_0(\D)'=H^{-2}(\D)$ is Hilbert-Schmidt:
	this is easily checked by considering a Fourier basis diagonalizing the Dirichlet Laplacian,
	recalling that $(1+\Delta)^2:H^2_0(\D)\to L^2(\D)$ is an isometry of Hilbert spaces
	and Weyl's law in dimension $2$.
	
	The proof is concluded since $H^{-2}(\D)$ is a Polish space and a topological subspace of $\S'(\D)$,
	and since $\M(\D)\subset H^{-2}(\D)$, again as topological vector spaces.
\end{proof}

In fact, regularity of $W$ discussed above might be improved 
with a finer Sobolev embedding (see \cite{Flandoli2018,Grotto2020d}).
For our purposes, it suffices to identify $(\XX,m)$
with a standard Borel probability space contained (topologically) in $\S'(\D)$,
so to have both a well-behaved underlying probability space and a natural way to couple samples
of $M$ and smooth functions.

\subsection{Double Stochastic Integrals and Moments}\label{ssec:double}
Having identified the random measure $M\sim [a,q,\nu]$ with a random distribution on $\D$,
one can consider $M\otimes M\in \S'(\D^2)$, its tensor product defined sample-by-sample.
In particular, for $h\in \S (\D^2)$ we can consider the duality coupling $\brak{h,M\otimes M}$.
By symmetry, coupling of $M\otimes M$ with $h(x,y)$ is the same as with $\tilde h(x,y)=(h(x,y)+h(y,x))/2$,
so one can reduce the discussion to symmetric functions.

From the point of view of Probability Theory, this is a rather unnatural definition for
double integrals with respect to $M$, and one typically prefers to use double stochastic integrals, 
which are classically defined as follows (see \cite[Chapter 5]{PeccatiTaqqu10} for a full account).

Let $M\sim [a,q,\nu]$ be defined on the probability space $(\Omega,\F,\PP)$. For any pair of disjoint sets $A, B\in \mathcal B$, we consider the mapping $(x,y)\mapsto h(x,y) = {\bf 1}_A(x) {\bf 1}_B(y) $, $x,y\in \D^2$, and define the \emph{ double stochastic integral of $h$ with respect to $M$} as	
\begin{equation*}
I^2_M(h) := 2 M(A)M(B).
\end{equation*}
One canonically extends by linearity the previous definition of $I^2_M(h)$ to functions $h$ that are generic finite linear combinations of simple functions as above. 
For any such $h$ it is easy to verify that, writing explicitly $M = adx +\sqrt{q} W+P$,
\begin{gather*}
	\label{eq:intsymm}
	I^2_M(h) = I^2_M(\tilde h)\\
	\label{eq:intexpanded}
	I^2_M(h)=a^2\int_{\D^2}\tilde hd^2x+2a I^1_M\pa{\int_{\D}\tilde h(x,\cdot)dx}+q I^2_W(\tilde h)+2\sqrt q I^1_WI^1_P(\tilde h)+I^2_P(\tilde h),
\end{gather*}
where:
\begin{itemize}
\item[--] $d^2 x$ stands for the product Lebesgue measure on $\D^2$;
\item[--] $I^2_W$ and $I^2_P$ are, respectively, double Wiener-It\^o integrals with respect to $W$ and $P$;
\item[--] the (slightly clumsy, but convenient for our further discussion) symbol $I^1_WI^1_P$ denotes a \emph{ mixed iterated} stochastic integral of order $2$, with respect to $W$ and $P$. Note that $I^1_WI^1_P(\tilde h) = I^1_PI^1_W(\tilde h)$ because of symmetry. See e.g. \cite{NS} for a discussion of several distinguished properties of mixed iterated stochastic integrals.
\end{itemize}

The following result is standard, and the proof is therefore omitted.

\begin{lemma}
	The mapping $h\mapsto I_M^2(h)$ defined above extends by density to a 
	continuous mapping $I^2_M:L^2(\D)\to L^2(\Omega,\F,\PP)$ between Hilbert spaces, 
	such that \eqref{eq:intsymm} and \eqref{eq:intexpanded} are preserved. 
\end{lemma}

We refer to \cite[Proposition 3.1]{Szulga1991} further details, as well as
to \cite{Nualart1995, NouPec12,PeccatiTaqqu10} for a discussion of the purely Gaussian setting, 
and to \cite{KW, LastPenrose18} for the Poissonian one.

Comparing definitions, one easily obtains the following relation between the
two notions of double integrals discussed above: for $h\in \S(\D^2)$,
\begin{equation}\label{eq:relation}
	\brak{h,M\otimes M}=I^2_M(h)+\int_{\D}h(x,x)dx+\int_\D\int_\R \gamma^2 \tilde h(x,x) \Nsc(dx,d\gamma).
\end{equation}

We eventually observe that moments of stochastic integrals can be evaluated by means of usual formulae (see again \cite{PeccatiTaqqu10}):
for $f\in L^2(\D)$ and $h\in L^2(\D^2)$ it holds
\begin{gather}\nonumber
\expt{I^1_W(f)}=\expt{I^1_P(f)}=\expt{I^2_W(h)}=\expt{I^2_P(h)}=0,\\
\expt{I^1_W(f)^2}=\norm{f}_{L^2(\D)}^2,\quad \expt{I^1_P(f)^2}=\int_\R x^2 d\nu(x) \norm{f}_{L^2(\D)}^2,\\
\expt{I^2_W(h)^2}=2 \norm{h}_{L^2(\D)}^2,\quad
\expt{I^2_P(h)^2}=2 \pa{\int_\R x^2 d\nu(x)}^2 \norm{h}_{L^2(\D)}^2
\end{gather}
(recall that we are assuming that $\Nsc$ is compensated).

\subsection{Characterizing Random Measures}\label{ssec:characterize}

%\todo[inline]{il risultato di Rajput e Rosinski dice che le misure infinitamente divisibili e completamente random
%hanno $\nu$ senza atomi (vedi loro def. di misura di levy a pag 5 del loro paper). 
%Sul $\nu(0)=0$ ho aggiunto un commento sopra, e anche sull'ipotesi $\sigma$-finita.}

Before moving on to Euler's dynamics, let us argue how the random measures we are considering are special cases of the objects satisfying \cref{def:rm} and \cref{hyp:main}.

A random measure is called \emph{infinitely divisible} if all random variables $M(A)$, $A\in\B$,
are infinitely divisible. This condition is clearly implied by \cref{hyp:main} b)
and these random measures are characterized in \cite[Proposition 2.4]{Rajput1989}
by means of their characteristic function,
generalizing the L\'evy-Khintchine representation for real-valued infinitely divisible random variables.

Random measures satisfying \cref{hyp:main} a) are usually referred to
as \emph{independently scattered} or \emph{completely random}.
Assuming this condition and the requirement that all variables $M(A)$ have finite second moments,
the aforementioned result of \cite{Rajput1989} on infinitely divisible random measures
specializes to the following one: we refer to \cite[Section 4]{Peccati2008} for a thorough discussion.

\begin{proposition}\label{prop:levykhintchine}
	Let $M$ be a random measure satisfying \cref{hyp:main}.
	Assume moreover that $\expt{M(A)^2}<\infty$ for all $A\in\B$.
	Then the characteristic function of $M$ has the L\'evy-Khintchine form
	\begin{equation}\label{e:lk}
	\expt{e^{it M(A)}}=
	\exp\pa{i t a |A|-\frac{t^2}{2}q |A|+|A|\int_\R\pa{e^{itx}-1 -itx }d\nu(x)},\quad t\in \mathbb{R},
	\end{equation}
	where $a\in\R$, $q\in [0,\infty)$ and $\nu$ is an atomless measure on $\R$
	such that $\int_\R  |x| \vee |x|^2 d\nu(x)<\infty$.
	
	Conversely, given such a \emph{triple} $a,q,\nu$,
	there exist a random measure satisfying \cref{hyp:main}
	with the above characteristic function.
\end{proposition}

It is easily seen that random measures of the type \eqref{eq:levyito} 
correspond to those characterized by \cref{e:lk}. 
Moreover, given $M$ as in \cref{prop:levykhintchine}, 
it is possible to define Gaussian and compound Poisson processes $W,P$
as above such that $M=a+\sqrt q W+P$ \emph{without changing probability space},
and such representation is unique:
this is a particular case of \cite[Theorem 4.5]{pedersen2003levy},
and we also refer to \cite{Adler1983,Sato2004} for details on the topic.

%%%%%%%%%%%%%%%%%%%%%%%%%%%%%%%%%%%%%%%%%%%%%%%%%%%%%%%%%%%%%%%%%%%%%%%%%%%%%%
\section{Weak Solutions of 2-dimensional Euler Equations}\label{sec:weaksol}

Let us first consider a smooth solution $\omega(t,x)$ of \cref{eq:eulero}.
The weak (\emph{i.e.} integral) formulation of \cref{eq:eulero} reads,
for all test functions $\varphi\in\S(\D)$,
\begin{equation*}
	\int_{\D}\varphi(x) \omega_t(x) dx-\int_{\D}\varphi(x) \omega_0(x) dx
	=\int_0^t \int_{\D^2} K(x,y)\cdot \nabla\omega_s(x)\omega_s(y)dxdy.
\end{equation*}
Since $G(x,y)$ is a symmetric function of its variables, $K(x,y)$ is in turn antisymmetric:
the latter expression can thus be rewritten as
\begin{gather}\label{eq:delort}
	\int_{\D}\varphi(x) \omega_t(x) dx-\int_{\D}\varphi(x) \omega_0(x) dx
	=\int_0^t \int_{\D^2}  H_\varphi(x,y)\omega_s(x)\omega_s(y)dxdy,\\ \nonumber
	H_\varphi(x,y)=\frac12 \pa{\nabla\varphi(x)-\nabla\varphi(y)}\cdot K(x,y).
\end{gather}
The fundamental advantage of this second, symmetrized form, is that instead of the singular kernel $K$
it involves the new symmetric function $H_\varphi$ (depending on the test function), which is bounded and smooth outside
the diagonal set $\set{x,y\in\D:x=y}$, on which it has a jump discontinuity, regardless of the choice of $\varphi$.
In fact, one can provide an explicit bound:
\begin{equation}\label{eq:H}
	\exists C>0: \,\forall \varphi\in\S(\D), \quad \norm{H_\varphi}_\infty\leq C \norm{\varphi}_{C^2(\D)}.
\end{equation}
This is readily seen by Taylor expanding $\nabla\varphi$, keeping in mind that $|K(x,y)|=O(|x-y|^{-1})$ for $|x-y|\to 0$.
Clearly, formulation \cref{eq:delort} is equivalent to \cref{eq:eulero} for smooth solutions.

The double space integral of \cref{eq:delort} can be defined for $\omega(t,x)$ whose space regularity is
lower than the well-posedness regime $L^\infty(\D)$. We refer to \cite{Delort1991,Schochet1995} for the original
application of this trick to Euler flows with low space regularity, and a more thorough discussion of the
symmetrization argument outlined above.

Aiming to define solutions starting from a \emph{random} initial datum whose law is a random field $M\sim [0,q,\nu]$,
considering \cref{eq:delort} sample-by-sample is still insufficient.
This is essentially due to the Gaussian part: as first observed in \cite{Flandoli2018}, 
$W\otimes W\in\S'(\D^2)$ can not be coupled with discontinuous functions, thus 
$H_\varphi$ is not regular enough to be coupled sample-by-sample with $W\otimes W$,
due to the already mentioned discontinuity on the diagonal.

We thus need to resort to stochastic integrals in order to define solutions of \cref{eq:delort}
for which $\omega(t,\cdot)$ is a random field with the law of $M\sim [0,q,\nu]$.
The idea is simple: we interpret all space integrals of \cref{eq:delort}
as stochastic integrals with respect to the random vorticity $\omega_t$.
We refer once again to \cite{Flandoli2018,Grotto2020d} for a discussion on why this gives a 
sensible definition of weak solution, in particular obtaining this kind of weak solutions
as limits of classical ones.

\subsection{Weak Solutions and Stochastic Integrals}\label{ssec:weaksol}

The following definitions are analogous to the ones given in the aforementioned \cite{Flandoli2018,Grotto2020d}
in the purely Gaussian or Poissonian case: in fact they are a generalization.
We write $(\XX,m)$ denoting the law of the restriction of $M\sim [a,q,\nu]$ to $\S(\D)$ on its support $\XX\subset \S'(\D)$.

\begin{definition}\label{def:sol}
	Let $I\subset \R$ be an interval and fix a probability space $(\Omega,\F,\PP)$.
	Consider a $\S'(\D)$-valued stochastic processes $(\omega_t)_{t\in I}$ satisfying
	\begin{enumerate}
		\item the fixed-time marginals have law $\omega_t\sim [a,q,\nu]$;
		\item for any $\varphi\in\S(\D)$, $\PP$-almost surely, the maps
		\begin{equation*}
			I\ni t\mapsto I^1_{\omega_t}(\varphi), \quad I\ni t\mapsto I^2_{\omega_t}(H_\varphi),
		\end{equation*}
		are measurable.
	\end{enumerate}
	We call $(\omega_t)_{t\in I}$ a weak solution of Euler's equations if for all $\varphi\in\S(\D)$, $\PP$-almost surely, for all $t,s\in I$,
	\begin{equation}\label{eq:eulerweak}
	I^1_{\omega_t}(\varphi)-I^1_{\omega_s}(\varphi)=\int_s^t I^2_{\omega_r}(H_\varphi) dr.
	\end{equation}
\end{definition}

Although we use the term ``stochastic process'', solutions are actually deterministic in nature,
randomness coming only from an initial datum distributed according to the invariant measure $M$.
No external noise is involved, 
and the whole forthcoming discussion will concern transformations of the standard Borel probability space
$(\XX,m)$, acting both as probability and phase space (in the spirit of Statistical Mechanics).

\begin{remark}\label{rmk:integrability}
	\cref{def:sol} actually implies some more properties of a weak solution $(\omega_t)_{t\in I}$.
	For instance, the processes of item (2) must also be, $\PP$-almost surely, locally square-integrable in time:
	for a fixed $\varphi\in\S$,
	\begin{equation*}
		\expt{\int_0^T I^1_{\omega_t}(\varphi)^2dt}=\int_0^T \expt{I^1_{\omega_t}(\varphi)^2}dt\leq C T \norm{\varphi}^2_{L^2(\D)},
	\end{equation*}
	with $C>0$ depending only on $q,\nu$, and the analogue for $I^2_{\omega_t}(H_\varphi)$.
	Swapping integrals follows from the Fubini-Tonelli theorem.
\end{remark}

\begin{remark}
	Assume that $M$ is purely Gaussian, \emph{i.e.} a multiple of $W$.
	The stochastic integral $I^1_{\omega_t}(\varphi)$ appearing on the left-hand side of \cref{eq:relation} 
	is then an element of the first Wiener chaos of the Gaussian space associated to $\omega_t\sim \sqrt q W$.
	On the right-hand side, however, it appears the double stochastic integral $I^2_{\omega_r}(H_\varphi)$, 
	belonging to the second chaos of the Gaussian space associated to $\omega_r\sim \sqrt q W$.
	
	If $\omega_t$ was a Gaussian process,
	that is its marginal distributions at different times are jointly Gaussian,
	this would make \cref{eq:eulerweak} an equation between non-trivial elements of different chaoses.
	It is thus important to notice that we only prescribed the law of $\omega_t$ at a single fixed time,
	so there is no contradiction.
	In fact, this reasoning implies that a solution $(\omega_t)_{t\in I}$ with marginals $\omega_t\sim [0,q,0]$
	is \emph{not} a Gaussian process.
\end{remark}

In the light of \cref{def:sol}, it is natural to consider the notion of \emph{solution flow}
for Euler's equations on $(\XX,m)$.

\begin{definition}\label{def:flow}
	A flow of weak solutions of Euler's equations in the sense of \cref{def:sol} is
	a one-parameter group $\T_t:(\XX,m)\to(\XX,m)$ of almost surely invertible, measurable, measure-preserving transformations
	of the probability space $(\XX,m)$ such that the stochastic process
	\begin{equation*}
	\omega_t=\T_t \omega_0, \quad \omega_0=M,
	\end{equation*}
	(defined on the probability space $(\XX,m)$) is a weak solution of Euler's equation in the sense of \cref{def:sol}.
\end{definition}

We already remarked in the Introduction that the existence of a flow in this sense remains an open question.
As anticipated, we now change our point of view in favor of observables of $(\XX,m)$,
rather than looking at samples of $\omega_t$.

\subsection{The Koopman-Von Neumann Point of View}\label{sec:kl}

It is a general fact, known as Koopman's Lemma,
that measure preserving flows give rise to groups of unitary operators on the $L^2$ space relative to the invariant measure.
In fact, there is a one-to-one correspondence between groups of unitaries and measure-preserving flows,
which is made precise in the following:

\begin{proposition}\label{prop:koopman}
	Let $(X,\F,\mu)$ be a standard Borel probability space,
	\emph{i.e.} $X$ is a Polish space and $\F$ the associated Borel $\sigma$-algebra.
	Consider maps 
	\begin{equation*}
		\forall t\in\R, \quad T_t:X\rightarrow X,\quad  U_t f=f\circ T_t, \quad f\in L^0(X,\F,\mu).
	\end{equation*}
	The following are equivalent:
	\begin{itemize}
		\item $T_t$ is a group of $\mu$-almost surely invertible, measurable, measure preserving maps,
		and $t\to T_t$ is weakly measurable:
		\begin{equation*}
			\forall A,B\in\F,\qquad t\mapsto \mu(T_t^{-1}A\cap B)\text{ is measurable;}
		\end{equation*}
		\item $U_t$ defines a strongly continuous group of unitary operators on $L^2(X,\F,\mu)$,
		such that for all $t\in\R$
		\begin{equation*}
			U_t 1=1,\qquad \forall f\geq 0,\, U_t f\geq 0.
		\end{equation*}
	\end{itemize}
\end{proposition}
\noindent
We refer to \cite{Goodrich1980} for the proof.
In the latter, and from now on, $L^p$ spaces are meant to be composed of 
\emph{complex-valued} functions.

Since $(\XX,m)$ is a standard Borel probability space thanks to \cref{cor:randomdistr},
a solution flow in the sense of \cref{def:flow} would thus correspond to a 
group of unitary operators on $L^2(\XX,m)$.
The crucial feature we care about is Euler's dynamics as prescribed in \cref{def:sol},
and in the operator setting this information is encoded in the infinitesimal generator.

Consider the class of \emph{cylinder observables}
\begin{equation*}
\cyl(M)=\set{F(M)=f(I^1_M(\varphi_1),\dots, I^1_M(\varphi_k)), \varphi_1,\dots,\varphi_n\in\S(\D), f\in C^\infty_b(\R^n,\C)}.
\end{equation*}
The latter expression for $F(M)$ clearly defines a measurable function of class $L^\infty(\XX,m)$.
\begin{lemma}\label{lem:cylinder}
	Let $M\sim [a,q,\nu]$. The linear subspace $\cyl(M)\subset L^2(\XX,m)$ is dense.
	The linear operator
	\begin{equation}\label{eq:generator}
	\A F(M)=\sum_{k=1}^{n} \de_k f(I^1_M(\varphi_1),\dots, I^1_M(\varphi_k)) I^2_M(H_{\varphi_k}), 
	\end{equation}
	defined on $F\in\cyl(M)$ takes value in $L^2(\XX,m)$.
	
	If $\T_t$ is a solution flow in the sense of \cref{def:flow}, the infinitesimal generator
	of the strongly continuous group of Koopman unitary operators $\U_t F=F\circ \T_t$, $F\in L^2(\XX,m)$,
	takes the form \eqref{eq:generator} when restricted to $\cyl(M)$.
\end{lemma}

\begin{proof}
Density of the class $C(M)$ is proved by means of the following two facts: 
(1) the $\sigma$-field generated by random variables of the type $I_M^1(\varphi)$ 
is the same as the $\sigma$-field generated by $M$, 
and (2) the class $C^\infty_b(\R^n,\C)$ contains arbitrary products of trigonometric functions. 
To conclude, one then classically exploits uniqueness of Fourier transforms 
-- \emph{e.g.} along the lines of \cite[proof of Theorem 2.2.4]{NouPec12}. 
The second statement is an easy consequence of the fact that double stochastic integrals in $M$ have finite second moments.
	
As for the last statement, a process $\omega_t$ given by a solution flow as in 
	\cref{def:flow} satisfies \cref{eq:eulerweak},
	so $F(\omega_t)$, $F\in\cyl(M)$, is almost surely differentiable at almost all times 
	(recall also \cref{rmk:integrability}) and
	\begin{equation*}
	\frac{d}{dt}F(\omega_t)=\sum_{k=1}^{n} \de_k f(I^1_{\omega_t}(\varphi_1),\dots, I^1_{\omega_t}(\varphi_k))I^2_{\omega_t}(H_{\varphi_k})
	=\A F(\omega_t).\qedhere
	\end{equation*}
\end{proof}

The following is the main result of the present contribution.

\begin{theorem}\label{thm:main}
	For all observables $F\in\cyl(M)$ of the form above, it holds
	\begin{equation*}
		\int_\XX \A F(M) d m(M)=0.
	\end{equation*}
	As a consequence, the densely defined operator $(\A,\cyl(M))$ is skew-symmetric.
\end{theorem}

In \cref{sec:conclusion} we will detail how from this one can deduce the existence of self-adjoint extensions
of $(\A,\cyl(M))$, and discuss related issues anticipated in the Introduction.

Before we move to the proof of \cref{thm:main}, we digress on the particular case in which $M\sim [0,0,\nu]$,
and $\nu$ is a finite measure,
which is both simpler and better behaved with respect to the general one,
and it also has a particular relevance in fluid-dynamics.

\subsection{Point Vortices}\label{ssec:pv}

If $M\sim [0,0,\nu]$ with $\nu$ a finite measure on $\R$,
then $M$ has the following representation:
letting $\nu=c\sigma$ with $\sigma$ a probability measure, there exist
(on the same probability space on which $M$ is defined)
a Poisson process $N_c$ of parameter $c$, i.i.d. uniform variables $x_1,x_2,\dots$ on $\D$
and i.i.d. real variables $\gamma_1,\gamma_2,\dots$ with law $\sigma$,
all these variables being independent,
such that $M=\sum^{N_c}_{i=0}\gamma_i\delta_{x_i}$.

Euler's evolution starting from an initial datum made of finitely many (weighted) Dirac's delta
becomes a finite-dimensional system, in which positions $x_i$ evolve but intensities $\gamma_i$ do not:
this is the so-called \emph{point vortex system},
for which we recall a classical well-posedness result.

\begin{proposition}\label{prop:pvondomain}
	Given any $\gamma_1,\dots, \gamma_N\in\R^\ast$, for almost all $(x_1,\dots, x_N)\in\D^N$
	with respect to the $N$-fold product Lebesgue measure, there exist a unique, global-in-time solution
	of the singular ODE system
	\begin{equation}\label{eq:pv}
	\dot x_i(t)=\sum_{j\neq i}\gamma_j K(x_i(t),x_j(t)),\quad x_i(0)=x_i.
	\end{equation}
	Such solutions define a measurable, measure-preserving flow $T_t:\D^N\to\D^N$.
\end{proposition}

The proof can be obtained with minor modifications to the one given in \cite[Chapter 2]{Marchioro1984}
for the case of the disk $\D=\set{x\in\R^2:|x|<1}$, as discussed therein.
The above dynamics is meant to describe the evolution of a planar,
incompressible, inviscid fluid whose vorticity is concentrated at a finite number of points.
These points, called vortices, move under according to the above ODE dynamics,
each under the influence of the other ones, \emph{without self-interactions}.

There is of course a deep connection between point vortex dynamics and Euler's equations 
\cref{eq:eulero} in PDE form:
the former can be obtained as a limit of $L^\infty(\D)$ solutions of the latter,
as detailed by a well-established part of the related literature 
(see the recent work \cite{Butta2018} for complete references).
Our interest for this result stems from the fact that the empirical measure of vortices,
$\omega_t=\sum_{i=1}^{N}\gamma_i\delta_{x_i}$,
still satisfies Euler's equations if these are considered in the symmetrized form \cref{eq:delort},
provided that we set $H_\varphi(x,x)=0$ for all $x$, in order to encode the absence of self-interaction,
see \cite{Marchioro1994}.

With that being said, it is clear that our study in the case $M\sim [0,0,\nu]$, finite $\nu$, essentially reduces to the
point vortex problem. Indeed, each given sample of such $M$ can be regarded as the empirical measure
of a finite set of vortices, and the above observations immediately lead to:

\begin{corollary}\label{cor:pv}
	For almost all $P=\sum_{i=1}^{N_c}\gamma_i \delta_{x_i}$ with respect to the law of $P\sim [0,0,\nu]$,
	$\nu$ finite, by \cref{prop:pvondomain} we can define
	\begin{equation*}
		\T_t(P)=\T_t\pa{\sum_{i=1}^{N_c}\gamma_i \delta_{x_i}}=\sum_{i=1}^{N_c}\gamma_i \delta_{x_i(t)},
	\end{equation*}
	with $x_i(t)$ the solution of \cref{eq:pv} having $N=N_c$ and the $\gamma_i$'s prescribed by $P$,
	starting from $x_i(0)=x_i$. 
	Such $\T_t$ is a solution flow for Euler's dynamics preserving the random measure $P$ in the sense of \cref{def:sol}.	
\end{corollary}

\begin{proof}
	The number of vortices and their intensities are invariant for the dynamics of \cref{prop:pvondomain}.
	Hence, it suffices to condition on random variables $N_c$, $\gamma_i$: $x_i$'s are then uniform i.i.d. random variables
	on $\D$, so the dynamics \cref{eq:pv} preserves their law.
\end{proof}

\cref{prop:pvondomain} implies that for almost all configurations of vortices the evolution is unique,
because vortices never travel too close one to another: the ODE system \cref{eq:pv} is driven by a vector field
involving $K$, which is singular only when two or more $x_i$'s coincide.
This implies (as discussed in \cite{Albeverio2003a}) Markov uniqueness of extensions of the generator $\A$.

However, it is possible to exhibit configurations of vortices leading to collapse: we refer to \cite{Grotto2020} for details.
How these phenomena relate to uniqueness in the extension of the generator $\A$ on cylinder function
is an interesting problem, for which we refer to \cite{Albeverio2003a,Grotto2020a}
(see also \cref{sec:conclusion} for further remarks on general $M\sim[0,q,\nu]$).

Finally, we stress that the case in which $\nu$ is not a finite measure has a completely different behavior:
in the general case $M\sim[0,0,\nu]$, samples are \emph{infinite sums} of Dirac's deltas 
(with summable weights), so the problem has an infinite-dimensional nature, just as in the purely Gaussian case.
Indeed, existence of a solution flow for general $\nu$ remains an open problem.

%%%%%%%%%%%%%%%%%%%%%%%%%%%%%%%%%%%%%%%%%%%%%%%%%%%%%%%%%%%%%%%%%%%%%%%%%%%%%%%5
\section{Infinitesimal Invariance}\label{sec:koopman}

We will first prove \cref{thm:main} separately in the purely Gaussian and Poissonian cases,
by using techniques specific to those setting, namely Malliavin's integration by parts and Mecke's formula.
The latter will also be used to deal with the interaction between Gaussian and Poissonian parts.

Infinitesimal invariance of the Gaussian white noise was first proved in \cite{Albeverio1979,Albeverio1990} 
on the 2-dimensional torus,
and our computation generalizes theirs on arbitrary bounded domain.
We will need to resort to Fourier series, and the following observation:

\begin{remark}\label{rmk:intH}
	Since $G(x,y)$ vanishes if either $x$ or $y\in\partial D$,
	$\nabla_x G(x,y)$ is parallel to the outward normal $\hat n$ of $\partial D$.
	With this in mind, integration by parts shows that
	\begin{equation*}
		\forall \phi\in\S(\D), \qquad \int_{\D} H_\phi(x,y)dx=0.
	\end{equation*}
\end{remark}

\subsection{Fourier Decomposition and Gaussian Infinitesimal Invariance}
Denote by $\set{\phi_k}_{k\in\N}$ a complete orthonormal basis of $L^2(\D)$ 
diagonalizing $-\Delta$ with Dirichlet boundary conditions,
\begin{equation*}
	-\Delta \phi_k=\lambda_k \phi_k,\qquad \lambda_k\geq 0.
\end{equation*}
In what follows, imposing that $\lambda_k$'s are ordered increasingly
is not relevant. The forthcoming two Lemmata collect important computations we will rely on, as they concern expressions
naturally appearing when considering the Fourier expansion of $H_\phi$.

\begin{lemma}\label{lem:fourier1}
	For any eigenfunction $\phi_k$ and $y\in D$ it holds
	\begin{equation*}
		\int_\D \phi_k(x)H_{\phi_k}(x,y)dx=0.
	\end{equation*}
\end{lemma}

\begin{proof}
	Integrating by parts, we have:
	\begin{align*}
	&2\int_\D \phi_k(x)H_{\phi_k}(x,y)dx= \int_\D \phi_k(x) (\nabla\phi_k(x)-\nabla\phi_k(y))\cdot K(x,y)dx\\
	&\qquad =\int_D \frac 12 \nabla (\phi_k(x))^2 \cdot K(x,y)dx-\nabla \phi_k(y)\cdot \int_D \phi_k(x)K(x,y)dx\\
	&\qquad =- \frac 12\int_D \phi_k(x)^2 \div_x K(x,y)dx+\nabla \phi_k(y)\cdot \int_D \nabla^\perp \phi_k(x)G(x,y)dx,\\
	\end{align*}
	from which we conclude recalling that $\div K=0$ and 
	\begin{equation*}
		\int_D \nabla^\perp \phi_k(x)G(x,y)dx=\frac1{\lambda_k}\nabla^\perp \phi_k(y),
	\end{equation*}
	the latter being orthogonal to $\nabla \phi_k(y)$ at all points $y\in D$.
\end{proof}
 
\begin{lemma}\label{lem:fourier2}
	Define, for $h,k,\ell\in\N$,
	\begin{equation*}
		C_{h,k,\ell}=\int_D\int_D H_{\phi_\ell}(x,y)\phi_k(x)\phi_h(y)dxdy.
	\end{equation*}
	It holds that:
	\begin{itemize}
		\item if any two indices coincide, $C_{h,k,\ell}=0$;
		\item for all $h,k,\ell\in\N$, $C_{h,k,\ell}=C_{k,h,\ell}$, that is: $C$ is symmetric in the first two indices;
		\item for all $h,k,\ell\in\N$,
		\begin{equation}\label{eq:cyclicsymmetric}
		C_{h,k,\ell}+C_{\ell,h,k}+C_{k,\ell,h}=C_{h,\ell,k}+C_{k,h,\ell}+C_{\ell,k,h}=0,
		\end{equation}
		that is, $C_{h,k,\ell}$ is cyclic symmetric in its indices.
	\end{itemize}
\end{lemma}

\emph{En passant}, we recall that cyclic symmetric functions are one of the three symmetry classes
partitioning the set of functions of three variables, 
the other classes being the symmetric and antisymmetric functions, \cite{Metropolis1991}.
 
\begin{proof}
	Expanding the definition of $H_{\phi_\ell}$, recalling that $K(x,y)=\nabla^\perp_x G(x,y)$
	and integrating by parts in $x$, we obtain
	\begin{align*}
		C_{h,k,\ell}
		&=\frac12 \int_D\int_D \pa{\nabla\phi_\ell(x)-\nabla\phi_\ell(y)}\nabla_x^\perp G(x,y)\phi_k(x)\phi_h(y)dxdy\\
		&=\pa{\frac1{\lambda_h}-\frac1{\lambda_k}} \int_D \nabla\phi_\ell(x)\nabla^\perp \phi_k(x) \phi_h(x)dx.
	\end{align*}
	If any two indices coincide, the latter expression readily shows that $C_{h,k,\ell}=0$:
	either the difference of eigenvalues is null or the integrand vanishes due to orthogonality.
	
	Starting from the integral in $dx$ only we just obtained to represent $C_{h,k,\ell}$,
	integrating by parts we have
	\begin{align*}
		C_{\ell,h,k}&=\pa{\frac1{\lambda_\ell}-\frac1{\lambda_h}} \int_D \nabla\phi_\ell(x)\nabla^\perp \phi_k(x) \phi_h(x)dx,\\
		C_{k,\ell,h}&=\pa{\frac1{\lambda_k}-\frac1{\lambda_\ell}} \int_D \nabla\phi_\ell(x)\nabla^\perp \phi_k(x) \phi_h(x)dx,
	\end{align*}
	which suffices to prove that the right-hand side of \cref{eq:cyclicsymmetric} vanishes.
	Noticing that $C_{h,k,\ell}$ is invariant under exchange of the first two indices,
	which is clear from the definition, concludes the proof
	of \cref{eq:cyclicsymmetric}, since this symmetry transforms one member of the equation into the other.
\end{proof}

Let us introduce a subclass of cylinder functions depending only on couplings with elements of the Fourier basis $\set{\phi_k}_{k\in\N}$,
\begin{equation*}
\cyl_F(W)=\set{F(W)=f(W(\phi_1),\dots, W(\phi_k)), n\in\N, f\in C^\infty_b(\R^n;\C)}.
\end{equation*}
We first prove infinitesimal invariance on these observables, and then extend by density
to the full $\cyl(W)$.

\begin{lemma}\label{lem:gaussinv}
	For all $F\in\cyl(W)$ it holds
	\begin{equation*}
	\expt{\A F(W)}=0.
	\end{equation*}
\end{lemma}

\begin{proof}
	Consider first $F\in\cyl_F(W)$: in this case we can apply directly Gaussian integration by parts,
	\begin{multline*}
	\expt{\sum_{k=1}^n\de_k f(W(\phi_1),\dots, W(\phi_k))I^2_W\pa{H_{\phi_k}}}\\
	=\expt{\int_D D_x F(W) \pa{\sum_{k=1}^n \phi_k(x) I^2_W\pa{H_{\phi_k}}}dx}\\
	=\expt{F(W) D^\ast\pa{\sum_{k=1}^n \phi_k(x) I^2_W\pa{H_{\phi_k}}}},
	\end{multline*}
	denoting by $D$ the Malliavin derivative and $D^\ast$ its adjoint (see e.g. \cite[Chapter 2]{Nualart1995}).
	Gaussian divergence in the right-hand side is given by
	\begin{multline*}
	D^\ast \pa{\sum_{k=1}^n \phi_k(x) I^2_W\pa{H_{\phi_k}}}\\
	=\sum_{k=1}^n I^1_W(\phi_k) I^2_W\pa{H_{\phi_k}} -2 \sum_{k=1}^n \int_D \phi_k(x) I^1_W\pa{H_{\phi_k}(x,\cdot)}dx.
	%&\qquad=\sum_{k=1}^n I^3_W(\phi_k\otimes H_{\phi_k}),
	\end{multline*}
	We now notice that 
	\begin{equation*}
	\int_D \phi_k(x) I^1_W\pa{H_{\phi_k}(x,\cdot)}dx=I^1_W\pa{\int_D \phi_k(x) H_{\phi_k}(x,\cdot)dx}=0
	\end{equation*}
	thanks to \cref{lem:fourier1}, so the second part of the divergence vanishes.
	Moreover, if $\F_n=\sigma\pa{I^1_W(\phi_1),\dots,I^1_W(\phi_n)}$ is the $\sigma$-algebra
	generated by standard Gaussian variables $I^1_W(\phi_1),\dots,I^1_W(\phi_n)$,
	$F(W)$ is $\F_n$-measurable, so conditioning on $\F_n$ we obtain
	\begin{equation*}
	\expt{\A F(W)}=\expt{F(W) \sum_{k=1}^n I^1_W(\phi_k) \expt{I^2_W\pa{H_{\phi_k}}\mid \F_n}}.
	\end{equation*}
	We conclude by proving that the inner sum in the last expression is almost surely null.
	Expanding the function of two variables $H_{\phi_\ell}(x,y)$ in Fourier basis,
	its coefficients are given by $C_{h,k,\ell}$ of \cref{lem:fourier2}, thus
	we now write
	\begin{align*}
	\expt{I^2_W\pa{H_{\phi_k}}\mid \F_n}
	&=\expt{\left.I^2_W\pa{\sum_{k,h\in\N}\phi_k(x)\phi_h(y) C_{h,k,\ell}}\right| \F_n},\\
	\sum_{\ell=1}^n I^1_W(\phi_\ell) \expt{I^2_W\pa{H_{\phi_\ell}}\mid \F_n}
	&=\sum_{h,k,\ell=1}^n I^1_W(\phi_\ell)I^1_W(\phi_k)I^1_W(\phi_h) C_{h,k,\ell},
	\end{align*}
	where the very last triple sum vanishes almost surely:
	indeed, each summand is the product of $I^1_W(\phi_\ell)I^1_W(\phi_k)I^1_W(\phi_h)$,
	which is symmetric in the indices $h,k,\ell$, and $C_{h,k,\ell}$, which is cyclic symmetric by \cref{lem:fourier2}.
	This concludes the thesis for $F\in \cyl_F(W)$.
	
	Consider now, for $m\in\N$,
	\begin{equation*}
		F(W)=f(I^1_W(\varphi_1),\dots,I^1_W(\varphi_1)),\quad F_m(W)=f(I^1_W(\Pi_m\varphi_k),\dots,I^1_W(\Pi_m\varphi_1)),
	\end{equation*}
	where $\Pi_m \varphi$ denotes the projection of $\varphi$ on the linear span of $\phi_1,\dots \phi_m$.
	We will also write $\Pi_m^\perp$ for the projection on the orthogonal space.
	We already proved the thesis for observables $F_m\in\cyl_F(W)$, so the proof is concluded if we show that
	$\A F-\A F_m$ converges to $0$ in $L^1(\Omega,\F,\PP)$.
	Using the definition of $\A$ we can expand
	\begin{align*}
		\A F(W)-\A F_m(W)
		&=\sum_{k=1}^n \partial_k f(I^1_W(\varphi_1),\dots)\bra{I^2_W(H_{\varphi_k})-I^2_W(H_{\Pi_m\varphi_k})}\\
		&\quad + \sum_{k=1}^n 
		\bra{\partial_k f(I^1_W(\varphi_1),\dots)-\partial_k f(I^1_W(\Pi_m\varphi_1),\dots)}I^2_W(H_{\Pi_m\varphi_k})\\
		&:=(1)+(2),
	\end{align*}
	and show that each of the two sums on the right-hand side go to $0$ in $L^1(\Omega,\F,\PP)$ as $m\to 0$.
	Since derivatives of $f$ are uniformly bounded, we have
	\begin{equation*}
		\expt{(1)^2}
		\leq n \norm{f}_{C^1}^2 \sum_{k=1}^n \expt{\pa{I^2_W(H_{\varphi_k})-I^2_W(H_{\Pi_m\varphi_k})}^2},
	\end{equation*}
	in which every summand vanishes as $m\to\infty$, since
	\begin{multline*}
		\expt{\pa{I^2_W(H_{\varphi_k})-I^2_W(H_{\Pi_m\varphi_k})}^2}
		=\expt{I^2_W(H_{\Pi_m^\perp\varphi_k})^2}=\norm{H_{\Pi_m^\perp\varphi_k}^2}_{L^2(\D)}^2\\
		\leq |\D|^2 \norm{H_{\Pi_m^\perp\varphi_k}^2}_\infty^2\leq C |\D|^2 \norm{\Pi_m^\perp\varphi_k}_{C^2(\D)}^2
		\xrightarrow{m\to\infty} 0,
	\end{multline*}
	where the first step uses linearity of $H_\varphi$ in $\varphi$, the second uses It\^o isometry,
	the fourth \cref{eq:H} and the conclusion follows since $\varphi$ is smooth
	(thus it is approximated by Fourier truncations in all Sobolev norms, and by Sobolev embedding in $C^k$ norms).
	As for the second sum, by Taylor expanding derivatives of $f$ we have
	\begin{equation*}
		\expt{ | (2) |}\leq \norm{f}_{C^2} \sum_{k=1}^n \expt{| I^1_W(\Pi_m^\perp\varphi_k) I^2_W(H_{\Pi_m\varphi_k})|},
	\end{equation*}
	in which every summand vanishes as $m\to\infty$, by virtue of the Cauchy-Schwarz inequality and of the relations
		\begin{align*}
		\expt{I^1_W(\Pi_m^\perp\varphi_k)^2}&=\norm{\Pi_m^\perp\varphi_k}^2_{L^2(\D)}\xrightarrow{m\to\infty} 0,\\
		\expt{I^2_W(H_{\Pi_m\varphi_k})^2} &\leq C |\D|^2 \norm{\Pi_m\varphi_k}_{C^2(\D)}^2
		\leq C |\D|^2 \norm{\varphi_k}_{C^2(\D)}^2,
	\end{align*}
where the second estimate is deduced by rehearsing arguments similar to those given above.	
\end{proof}

\subsection{Poissonian Infinitesimal Invariance}

We consider first invariance of a purely Poissonian random measure.

\begin{lemma}\label{lem:poissoninv}
	If $P\sim[0,0,\nu]$, then for all $F\in\cyl(P)$ it holds
	\begin{equation*}
	\expt{\A F(P)}=0.
	\end{equation*}
\end{lemma}

\begin{proof}
	Applying the bivariate Mecke's formula \cite[Theorem 4.4]{LastPenrose18} applied to the Poisson point process
	$\Nsc$ on $\R\times \D$ with intensity $d\nu(\xi)\otimes dx$,
	we infer that
	\begin{multline*}
		\expt{\sum_{k=1}^{n} \de_k f(I^1_P(\varphi_1),\dots, I^1_P(\varphi_k)) I^2_P(H_{\varphi_k})}=\\
		\int_{\R^2}\int_{\D^2} \sum_{k=1}^{n}\expt{\de_k f(I^1_P(\varphi_1)+\xi\varphi_1(x)+\xi'\varphi(x'),\dots)}\\
		\cdot \xi\xi' H_{\varphi_k}(x,x')d\nu(\xi)d\nu(\xi')dxdx'.
	\end{multline*}
	Conclusion now follows from integration by parts on $\D$.
	Observe first that
	\begin{equation}\label{eq:nablaf}
		\nabla_x F(P+\xi\delta_x+\xi'\delta_{x'})
		=\sum_{k=1}^{n}\de_k f(I^1_P(\varphi_1)+\xi\varphi_1(x)+\xi'\varphi(x'),\dots)
		\xi \nabla\varphi_k(x),
	\end{equation}
	so combining the latter, the definition of $H_\varphi$ and Mecke's formula above we obtain
	\begin{multline*}
	\expt{\sum_{k=1}^{n} \de_k f(I^1_P(\varphi_1),\dots, I^1_P(\varphi_k)) I^2_P(H_{\varphi_k})}=\\
	\frac12 \int_{\R^2}\int_{\D^2} 
	\expt{\nabla_x F(P+\xi\delta_x+\xi'\delta_{x'})-\nabla_{x'} F(P+\xi\delta_x+\xi'\delta_{x'})}\\
	\cdot K(x,x')d\nu(\xi)d\nu(\xi')dxdx'.
 	\end{multline*}
	Integrating by parts in $x$ and $x'$ this last expression and recalling that $K(x,x')$
	is divergence-less in both variables, the proof is concluded.
\end{proof}

Combining invariance of the Gaussian and Poissonian parts, we deduce infinitesimal invariance on
cylinder functions.

\begin{lemma}\label{lem:invariance}
	Given any $M\sim [a,q,\nu]$, for all $F\in\cyl(M)$ it holds
	\begin{equation*}
	\expt{\A F(M)}=0.
	\end{equation*}
\end{lemma}

\begin{proof}
	Assume first that $a=0$.
	We distinguish contributions from the Gaussian and Poissonian parts, and exploit their independence.
	Starting from the definition of $\A$ in \cref{eq:generator} we expand the double stochastic integral according 
	to the decomposition $M=\sqrt q W+P$, 
	\begin{equation*}
		I^2_M(H_\varphi)=qI^2_W(H_\varphi)+2\sqrt{q}I^1_PI^1_W(H_\varphi)+I^2_P(H_\varphi).
	\end{equation*}
	By linearity and Lemmas \ref{lem:gaussinv} and \ref{lem:poissoninv}
	the thesis reduces to
	\begin{equation*}
		\expt{\A F(M)}=\expt{\sum_{k=1}^{n} \de_k f(I^1_M(\varphi_1),\dots, I^1_M(\varphi_k)) I^1_PI^1_W(H_{\varphi_k})}=0.
	\end{equation*}
	
	In order to apply Gaussian integration by parts we resort again to Fourier decomposition
	and assume $F\in\cyl_F(M)$, that is $\varphi_k=\phi_n$ are elements of the Fourier basis.
	We denote in the following by $D_W$ and $D_W^\ast$ the Malliavin derivative and
	Skorohod integral in the Gaussian space associated to $W$.
	For a fixed sample of the compound Poisson process $P$, the functions
	\begin{equation*}
		\sum_{k=1}^{n} \de_k f(I^1_P(\phi_1)+\sqrt q I^1_W(\phi_1),\dots)\phi_1(x),\quad
		\sum_{k=1}^{n} I^1_PI^1_W(H_{\varphi_k})\phi_1(x),
	\end{equation*}
	are cylinder functions of $W$, that is they belong to $\cyl_F(W)$,
	thus by conditioning on $P$ and Gaussian integration by parts (applied as in \autoref{lem:gaussinv}) we obtain
	\begin{multline*}
		\expt{\A F(M)}=\expt{\sum_{k=1}^n \de_k f (\sqrt q I^1_W(\phi_1)+I^1_P(\phi_1),\dots) 
			\sqrt qI^1_P I^1_W(H_{\phi_k})}\\
		= \expt{F(M) \sqrt q D^\ast \pa{\sum_{k=1}^n \phi_k I^1_P I^1_W(H_{\phi_k})}}.
	\end{multline*}
	The Gaussian divergence in the latter is given by
	\begin{align*}
	&D^\ast \pa{\sum_{k=1}^n \phi_k I^1_P I^1_W(H_{\phi_k})}\\
	&\quad = \sum_{k=1}^n I^1_W(\phi_k)I^1_P I^1_W(H_{\phi_k})-2 I^1_P\pa{\int_\D \phi_k(x)H_{\phi_k}(x,\cdot)dx},
	\end{align*}
	in which the deterministic integral vanishes because of \cref{lem:fourier1}, so
	\begin{equation*}
	\expt{\A F(M)}=\expt{F(M) \sum_{k=1}^n I^1_W(\phi_k) I^1_P I^1_W(H_{\phi_k})}.
	\end{equation*}
	By expanding in Fourier series $H_{\phi_k}$ with symbols $C_{h,k,\ell}$ defined in \cref{lem:fourier2},
	and conditioning on the $\sigma$-algebra generated by $I^1_W(\phi_1),\dots I^1_W(\phi_n)$
	and $I^1_P(\phi_1),\dots I^1_P(\phi_n)$ similarly to what we did in the proof of \cref{lem:gaussinv},
	we arrive at
	\begin{equation*}
	\expt{\A F(M)}=\expt{F(M) \sum_{h,k,\ell=1}^n I^1_W(\phi_\ell)I^1_W(\phi_k)I^1_P(\phi_h)C_{h,k,\ell}}.
	\end{equation*}
	We conclude once again with a symmetry argument relying on \cref{lem:fourier2}: 
	since $C_{h,k,\ell}$ is invariant under exchange of $h,k$,
	we can symmetrize the two indices, so that the inner sum becomes
	\begin{multline*}
	\sum_{h,k,\ell=1}^n I^1_W(\phi_\ell)I^1_W(\phi_k)I^1_P(\phi_h)C_{h,k,\ell}\\
	=\frac12 \sum_{h,k,\ell=1}^n I^1_W(\phi_\ell)\pa{I^1_W(\phi_k)I^1_P(\phi_h)+I^1_P(\phi_k)I^1_W(\phi_h)}C_{h,k,\ell},
	\end{multline*}
	and thus must vanish, since each summand is the product of the cyclic-symmetric $C_{h,k,\ell}$ and a symmetric function
	of the indices.

%	Since $W$ and $P$ are independent we can exchange expected values with
%	respect to their laws, so taking first expectation with respect to $P$
%	we can apply Mecke's formula \cite[Theorem 4.1]{LastPenrose18} and obtain
%	\begin{align*}
%		&\expt[P]{\sum_{k=1}^{n} \de_k f(I^1_M(\varphi_1),\dots) I^1_PI^1_W(H_{\varphi_k})}\\
%		&\quad=\int_\R \int_\D \expt[P]{\sum_{k=1}^{n}\de_k f(I^1_M(\varphi_1)+\xi\varphi_1(x),\dots)}
%			\xi I^1_W(H_{\varphi_k}(x,\cdot))d\nu(\xi)dx\\
%		&\quad=\frac12 \int_\R \int_\D \nabla_x \expt[P]{F(M+\xi\delta_x)}
%		\xi I^1_W(K(x,\cdot))d\nu(\xi)dx\\
%		&\qquad -\frac12 \int_\R \int_\D \expt[P]{\sum_{k=1}^{n}\de_k f(I^1_M(\varphi_1)+\xi\varphi_1(x),\dots)}
%		\xi I^1_W(\nabla\varphi_k(\cdot) K(x,\cdot))d\nu(\xi)dx,
%	\end{align*}
%	the second step making use of the definition of $H_\varphi$ and \cref{eq:nablaf}.
%	On the right-hand side, the first summand vanishes after integration by parts in $x$,
%	since $K(x,\cdot)$ is divergence-less.

	In order to remove the additional hypothesis $F\in\cyl_F(M)$ we can proceed by approximation
	as in \cref{lem:gaussinv}. For $m\in\N$ one considers, as we did above,
	the Fourier truncation $\Pi_m\varphi$ of a function $\varphi$, and shows that 
	the following difference vanishes in $L^2(\Omega,\F,\PP)$:
	\begin{align*}
		\sum_{k=1}^n \partial_k f(I^1_M(\varphi_1),\dots)I^2_M(H_{\varphi_k})-
		\sum_{k=1}^n \partial_k f(I^1_W(\Pi_m\varphi_1),\dots)I^2_M(H_{\Pi_m\varphi_k}).
	\end{align*}
	The proof actually proceeds in complete analogy with the one of \cref{lem:gaussinv},
	the only difference being that now mixed integrals appear in the estimates:
	this does not create any additional difficulty since one only uses isometric relations, that are the same for multiple and iterated integrals.
	
	Finally, we remove the additional hypothesis $a=0$. This is easily done thanks to the cancellation
	observed in \cref{rmk:intH}: if $M=a+\sqrt q W+ P$ it holds
	\begin{align*}
		I^2_{M}(H_\phi)&=a^2 \int_{\D^2} H_\phi(x,y)dxdy+ I^1_{\sqrt q W+ P}\pa{a\int_{\D} H_\phi(x,s)dx}
		+I^2_{\sqrt q W+ P}(H_\phi)\\
		&=I^2_{\sqrt q W+ P}(H_\phi),
	\end{align*}
	so one can always reduce to the case $a=0$ by considering the cylinder function of $M-a$ defined by $G(M-a)=F(M)$.
\end{proof}

We are thus only left to prove the last claim of \cref{thm:main}.

\begin{corollary}\label{cor:skewsymm}
	For $M\sim [a,q,\nu]$, the linear operator $\A$ defined on cylinder functions $\cyl(M)$ by \cref{eq:generator}
	is a skew-symmetric operator in $L^2(M)$.
\end{corollary}

\begin{proof}
	It suffices to observe that $\A$ acts on $\cyl(M)$ as a derivation, that is we have the following Leibnitz rule:
	for $F,G\in\cyl(M)$ also $FG\in\cyl(M)$ and
	\begin{equation*}
		\A(FG)=\A F\cdot G+ F\cdot\A G,
	\end{equation*}
	as it follows from direct inspection of the definition.
	From \cref{lem:invariance} it then follows
	\begin{equation*}
		\expt{F \A G}=\expt{\A(FG)}-\expt{G \A F}=-\expt{G \A F}.\qedhere
	\end{equation*}
\end{proof}

\begin{remark}
	The argument of \cref{cor:skewsymm} clearly applies also to 
	$(\A,\cyl_F(M))$, which is thus closable in $L^2_\C(\XX,m)$.
	The part of \cref{lem:gaussinv} extending the thesis from $\cyl_F(W)$ to $\cyl(W)$
	can then be understood as a proof that $\cyl(W)$ is contained in the domain
	of the closure of $(\A,\cyl_F(W))$.
\end{remark}

%%%%%%%%%%%%%%%%%%%%%%%%%%%%%%%%%%%%%%%%%%%%%%%%%%%%%%%%%%%%%%%%%%%%%%%%%%%%%%%5
\section{Concluding Remarks, Generalizations and Open Problems}\label{sec:conclusion}

We now briefly comment on point {\bf (c)} of the program outlined in \cref{ssec:flow},
and conclude the paper with some remarks on generalization and possible further developments.

\subsection{Self-Adjoint Extensions, Uniqueness and Solution Flows}
By \cref{cor:skewsymm}, the operator $(\A,\cyl(M))$ defined by \cref{eq:generator}
is closable in $L^2(\XX,m)$. We can extend $(\A,\cyl(M))$ to the generator of a group of unitaries on $L^2(\XX,m)$,
by means of the following results from Functional Analysis.

\begin{proposition}[Stone]\label{prop:stone}
	To each self-adjoint operator $(L,D(L))$ on a separable (complex) Hilbert space $H$ it is associated a unique
	strongly continuous one-parameter group of unitary operators $U_t:H\to H$, $t\in\R$ (usually denoted $U_t=e^{itL}$), such that
	\begin{equation*}
	\forall f\in D(L),\quad \frac{d}{dt}U_t f=i L f.
	\end{equation*}
\end{proposition}

\begin{proposition}[Von Neumann]\label{prop:vonneumann}
	Let $(L,D(L))$ be a densely defined symmetric operator on a separable (complex) Hilbert space $H$.
	Assume that there exists an antilinear, norm-preserving involution $J:H\to H$
	mapping $D(L)$ into itself and commuting with $L$ on $D(L)$.
	Then $L$ has at least a self-adjoint extension.
\end{proposition}

These results concern self-adjoint extensions of symmetric operators,
but in our setting $(\A,\cyl(M))$ is \emph{skew-symmetric}, so
in order to apply these results we need to consider $\tilde\A=-i \A$. 
The operator $\tilde\A$ is now a densely defined, symmetric operator on $L^2_\C(\XX,m)$,
and it commutes with
\begin{equation*}
L^2_\C(\XX,m)\ni F \mapsto JF \in L^2_\C(\XX,m), \quad JF(x)=\overline{F(-x)},\,x\in\XX,
\end{equation*}
which satisfies the hypothesis of \cref{prop:vonneumann}.
We conclude that there exists a self-adjoint extension of $\tilde\A$,
and by \cref{prop:stone} we can consider the associated group of unitaries $\U_t:L^2(\XX,m)\to L^2(\XX,m)$.
Here we encounter the first difficult open question
(which remains open also in the purely Gaussian and Poissonian case):

\begin{open}
	Is $(\tilde\A,\cyl(M))$ essentially self-adjoint? Equivalently, is the closure of $(\tilde\A,\cyl(M))$
	already a self-adjoint operator?
\end{open}

It is immediate to observe that $\U_t$ is unit-preserving: the constant $1$ belongs to $\cyl(M)$ and
\begin{equation*}
\U_0 1=\mbox{Id}1=1,\quad \frac{d}{dt}U_t 1=-i\A 1=0.
\end{equation*}
The only requirement left to satisfy the hypothesis of \cref{prop:koopman} and deduce that $\U_t$ is a Koopman group
associated to some group of transformations $\T_t$ on $(\XX,m)$,
thus obtaining a solution flow in the sense of \cref{def:flow}, is preservation of positivity.

\begin{open}
	Are there self-adjoint extensions of $(\tilde\A,\cyl(M))$ that generate a positivity preserving group of unitaries $\U_t$?
\end{open}

Indeed, in that case, by \cref{prop:koopman}, $\U_t$ would be the Koopman group of a one-parameter group $\T_t:(\XX,m)\to(\XX,m)$ of almost surely invertible, measurable, measure-preserving transformations, that is $\U_t F=F\circ \T_t$ for all $F\in L^2(\XX,m)$.
Existence of a solution flow in the sense of \cref{def:flow} would follow 
observing that transformations $\T_t$ produce solutions of Euler's dynamics in the sense
of \cref{def:sol}, see \cref{lem:cylinder}.

\subsection{Compact Surfaces Without Boundaries}
All the above arguments are easily repeated in the case where the domain $\D$ is replaced by a
compact surface without boundary, since the analytic (thus, functional analytic) setting is completely analogous.
We refer to \cite{Marchioro1994,Grotto2020b} for Euler's dynamics, both in the PDE case and point vortices model,
on surfaces such as the 2-dimensional torus and sphere, in particular to the second reference
for a comparison between those settings and the bounded domain case.

\subsection{About Point Vortices Approximations}
It should be possible to prove existence of solutions in the sense of \cref{def:sol} 
preserving a completely random measure $M\sim [a,q,\nu]$ as a scaling limit of point vortices systems,
that is considering the empirical measure of the particle system described in \cref{ssec:pv},
and taking a limit in which the number of vortices $N\to\infty$ and intensities of vortices are suitably rescaled.
This is the strategy of \cite{Flandoli2018,Grotto2020b} in the purely Gaussian case,
which we propose to investigate in our setting in future works.

One one hand, it might be possible to identify a suitably weak compactness criterion in spaces of distributions
so that finiteness of moments of $\nu$ is not necessary for compactness estimates,
leading to an even larger class of invariant measures.
On the other hand, building of solutions by means of such approximations clearly does not
imply any uniqueness result, and neither the flow property of \cref{def:flow}.

%\bibliography{isivbiblio}{}
\bibliographystyle{plain}

\end{document}